\documentclass[11pt]{amsart}
\usepackage[margin=1.5cm]{geometry}
\usepackage{mathrsfs}
\usepackage{epic,eepic,epsf,epsfig,tikz}
\usetikzlibrary{positioning}
\usepackage{amsmath}
\usepackage{amssymb}
\usepackage{amsbsy}
\usepackage{amsfonts}
\usepackage{color}

\numberwithin{equation}{section}
\newtheorem{lem}{Lemma}[section]%
\newtheorem{theorem}[lem]{Theorem}%
\newtheorem{defi}[lem]{Definition}%
\newtheorem{cor}[lem]{Corollary}%
\newtheorem{prob}[lem]{Problem}%
\newtheorem{prop}[lem]{Proposition}%

\def\a{\alpha}

  \def\G{\Gamma}

\def\nd{\mathrel{\bigm|\kern-.7em/}}

\def\f{\noindent}
\def\R{{\mathcal R}}
\def\Aut{\hbox{\rm Aut}}
\def\Cay{\hbox{\rm Cay}}

\def\BiCay{\hbox{\rm BiCay}}

\def\mz{{\mathbb Z}}

\newcommand{\sg}[1]{\langle #1 \rangle}

\begin{document}
\title[]{Symmetric property and edge-disjoint Hamiltonian cycles of the spined cube}

\author{Da-Wei Yang$^*$}
\address{Da-Wei Yang, School of Science, Beijing University of Posts and Telecommunications, Beijing, 100876, P.R. China}

\author{Zihao Xu}
\address{Zihao Xu, School of Computer Science, Beijing University of Posts and Telecommunications, Beijing, 100876, P.R. China}

\author{Yan-Quan Feng}
\address{Yan-Quan Feng, School of Mathematics and Statistics, Beijing Jiaotong University,
Beijing, 100044, P.R. China}

\author{Jaeun Lee}
\address{Jaeun Lee, Mathematics, Yeungnam University, Kyongsan 712-749, Republic of Korea}

\footnotetext[1]{Corresponding author. E-mail: dwyang@bupt.edu.cn}

\date{}
 \maketitle

\begin{abstract}
The spined cube $SQ_n$ is a variant of the hypercube $Q_n$, introduced by Zhou et al. in [Information Processing Letters 111 (2011) 561-567] as an interconnection network for parallel computing. A graph $\G$ is an $m$-Cayley graph if its automorphism group $\Aut(\G)$ has a semiregular subgroup acting on the vertex set with $m$ orbits,
and is a Caley graph if it is a 1-Cayley graph. It is well-known that $Q_n$ is a Cayley graph of an elementary abelian 2-group $\mz_2^n$ of order $2^n$. In this paper, we prove that $SQ_n$ is a 4-Cayley graph of $\mz_2^{n-2}$ when $n\geq6$, and is a $\lfloor n/2\rfloor$-Cayley graph when $n\leq 5$.
This symmetric property shows that an $n$-dimensional spined cube with $n\geq6$ can be decomposed to eight vertex-disjoint $(n-3)$-dimensional hypercubes, and as an application, it is proved that there exist two edge-disjoint Hamiltonian cycles
in $SQ_n$ when $n\geq4$.
Moreover, we determine the vertex-transitivity of $SQ_n$, and prove that $SQ_n$ is not vertex-transitive unless $n\leq3$.

\bigskip

\f {\bf Keywords:} Interconnection network, The spined cube, 4-Cayley graphs, Hamiltonian cycles.

\medskip
\f {\bf 2010 Mathematics Subject Classification:} 05C25, 20B25.
\end{abstract}

\section{Introduction}

An interconnection network, say network shortly, is the backbone of a parallel computing system, and connects the processors of the system. The computational cost of a parallel computing system is heavily dominated by the communication cost of the underlying network, which decides the overall performance of the system. This fact clearly emphasizes the significance of network topology and its efficient structural designs~\cite{ASD,X}.
The topological structure of the underlying network can be modeled as a graph where vertices correspond to processors, memory modules or switches, and edges correspond to communication links. It has been universally accepted and used by computer scientists and engineers~\cite{AC,X}.

\subsection{Symmetric property of networks}

In the design of a network, it is desirable that the designed network can provide us with high regularity and symmetry,
since it is advantageous to construction and simulation of some algorithms (see~\cite{X}). The class of vertex-transitive graphs possesses high regularity and symmetry,
and thus is an important and ideal class of topological structures of interconnection networks~\cite{AK,X,ZWYY}. A number of networks, including hypercubes~\cite{Bhu}, varietal hypercubes~\cite{WFZ}, balanced hypercubes~\cite{ZWYY}, and some of their generalizations are all vertex-transitive~\cite{X}.

A graph $\G$ is {\em vertex-transitive} if it looks the same when we take a view from every vertex~\cite{PCY}. The vertex-transitivity of graphs is usually measured by using group actions. Let $V(\G)$ and $E(\G)$ be the vertex set and edge set of $\G$, respectively. An {\em automorphism} of $\G$ is a permutation $\pi$ on $V(\G)$ satisfying the adjacency-preserving condition
$$(u, v)\in E(\G)~{\rm if\ and\ only\ if}\ (u^{\pi}, v^{\pi})\in E(\G).$$
The set of all automorphisms
of $\G$ forms a group under the operation of composition, denoted by $\Aut(\G)$,
and it is referred to as the {\em full automorphism group} of $\G$. A subgroup $G$ of $\Aut(\G)$ is {\em transitive} on $V(\G)$ if for any pair $(u,v)$ of vertices in $\G$ there is some $\pi\in G$ such that $v=u^{\pi}$.
(For a vertex $u$, the set $\{u^{\a}~|~\a\in G\}$ is an {\em orbit} of $G$ acting on $V(\G)$. The transitivity of $G$ on $V(\G)$ means that $G$
has exactly one orbit on $V(\G)$.)
A graph $\G$ is {\em vertex-transitive} if $\Aut(\G)$
is transitive on $V(\G)$.

The class of Cayley graphs presents a very useful graph-theoretic model for designing, analyzing, and improving symmetric networks~\cite{AK,he,ZWYY}. In particular, it plays an important role in constructing vertex-transitive graphs.
For a graph $\G$, a subgroup $G$ of $\Aut(\G)$ is {\em semiregular} on $V(\G)$ if evey element in $G$, except the identity, cannot fix a vertex of $\G$, and {\em regular} if $G$ is both transitive and semiregular on $V(\G)$. The graph $\G$ is a {\em Cayley graph} of a group $G$ if there exists a regular subgroup of $\Aut(\G)$ isomorphic to $G$ (see~\cite{CMY,X}).

The concept of Cayley graphs can be naturally generalized to $m$-Cayley graphs, where regular actions are replaced with semiregular actions. A graph $\G$ is said to be an {\em $m$-Cayley graph} of a group $G$ if $\Aut(\G)$ admits a semiregular subgroup isomorphic to $G$ having $m$ orbits on $V(\G)$. Of course, 1-Cayley graphs are simply Cayley graphs.
For additional results regarding $m$-Cayley graphs we refer the reader to~\cite{AT, HKM}. The class of $m$-Cayley graphs provides a useful tool to study non-vertex-transitive graphs, see~\cite{CZFZ} for example. It also has been used in the research of some networks, see~\cite{CMY,LM} for example.

The hypercube $Q_{n}$ is one of the most popular, versatile and efficient topological structures of networks~\cite{X}. Because of its many excellent features, it becomes the first choice for parallel processing and computing systems, one of which is its small diameter~\cite{Bhu,X}. Communication efficiency is a critical metric in a parallel computing system, while the diameter of a network is an important metric for communication
efficiency~\cite{ZFJZ}. A superior nature of the hypercube $Q_n$ is that its diameter is equal to its dimension $n$, which is
logarithm-level with respect to the order of $Q_n$.
To further improve the performance of the hypercube network in terms of diameter, numerous variant networks were put forward successively. The $n$-dimensional spined cube $SQ_n$ is
one variant of the hypercube, which was proposed by Zhou et al.~\cite{ZFJZ}. The diameter of $SQ_n$ is only $\lceil\frac{n}{3}\rceil+3$,
which is less than many known variants such as crossed cubes, twisted cubes, M\" obius
cubes, etc. The spined cube has attracted the attention of many researchers,
and its various properties such as embedability~\cite{ASD}, reliability~\cite{CZW}, the shortest-path routing~\cite{SKHB} have been investigated.

The symmetric property of the hypercube $Q_{n}$ have been widely investigated. It is a Cayley graph of an elementary abelian 2-group $\mz_2^n$, and consequently it is vertex-transitive. The symmetric properties of many variants of the hypercube have been studied, and however, there are also some variants whose
symmetric properties are not clear.
One may see a summary in Table~\ref{Table-1}.
The symmetric property of the spined cube is first considered in this paper. It is shown that an $n$-dimensional spined cube $STQ_n$
is a 4-Cayley graph of an elementary abelian 2-group $\mz_2^{n-2}$ when $n\geq6$,
and is a $\lfloor n/2\rfloor$-Cayley graph when $n\leq 5$.
Moreover, we determine the vertex-transitivity of $SQ_n$, and prove that $SQ_n$ is vertex-transitive only when $n\leq3$.

\begin{table}[h]

\begin{center}
\begin{tabular}{l|c|c|c|c}

\hline
Networks & Vertex-transitive & $m$-Cayley graph & $Orb$ &Reference \\
\hline
Hypercube $Q_n$ & Yes & $m=1$& 1& \cite{FZL}\\
\hline
Folded hypercube $FQ_n$ & Yes & $m=1$&1& \cite{FZL}\\
\hline
Balanced hypercube $BH_n$ & Yes & $m=1$&1& \cite{ZKFW}\\
\hline
Varietal hypercube $VQ_n$& Yes & $m=1$&1& \cite{WFZ}\\
\hline
Twisted cube $TQ_n$ & No & ? & ? &\cite{AP}\\
\hline
 Locally twisted cube $LTQ_n$ $(n\geq4)$&No & $m=2$&2&  \cite{CMY}\\
\hline
 Crossed cube $CQ_n$ $(n\geq5)$& No & ?  & ? &\cite{KB}\\
\hline
Folded crossed cube $FCQ_n$ $(n\geq5)$& No & ? & ? &\cite{PCY}\\
\hline
Twisted hypercube $H_n$ & ? & ? &?&\cite{Zhu}\\
\hline
Data center network $D_{k,n}$ $(k\geq2,n\geq2)$ &No& ? &?&\cite{Lv1}\\
\hline
Spined cube $SQ_n$ $(n\geq6)$& No & $m=4$ & ? & This paper\\
\hline
\end{tabular}
\end{center}
\vskip 0cm
\caption{{\small Summary of symmetric properties of some networks.}}\label{Table-1}
\end{table}

For a network $\G$, determining the number $Orb(\G)$ of orbits of $\Aut(\G)$
acting on vertices is an interesting problem in the study of symmetry.
This has been attracted the attention of some researchers. For example,
Pai et al.~\cite{PCY} put forward an open problem to determine the number of orbits for
the crossed cube and the folded crossed cube.
To a certain extent, determining the number $Orb(\G)$ is related to the $m$-Cayley property of networks. There is a famous conjecture that almost
all vertex-transitive graphs are Cayley graphs, see~\cite{MP} for the detail.
Moreover, the $m$-Cayley property of a network
is helpful for us to analyze its internal structure and other properties in some time. For example, the Cayley property and bi-Cayley property have been used to analyze the reliability of networks in~\cite{ZWYY} and~\cite{CMY}, respectively.
One may also see~\cite{LM} for other work.
What's more, the 4-Cayley graphic structure of $SQ_n$ with $n\geq6$ obtained in this paper will be applied to construct the edge-disjoint Hamiltonian cycles in $SQ_n$ in Section 4. These, combined with Table~\ref{Table-1}, prompt us to consider the following problem.

\begin{prob}\label{prob=1}
For some variants of hypercubes, including the crossed cube, the folded crossed cube
and the twisted cube,
\begin{enumerate}
\item [\rm (1)] determining the numbers of orbits of their automorphism groups acting on vertices;
\item [\rm (2)] determining the minimum number $m$ such that the variant network is an $m$-Cayley graph.
\end{enumerate}
\end{prob}

We note that the answers of Problem~\ref{prob=1} for the spined cube $SQ_n$
are both $4$ when $n\geq6$.
We also determine the full automorphism group of the spined cube based on its 4-Cayley property. Since the proof involves
some combinatorial group theory, it will be presented in other place.

\subsection{Edge-disjoint Hamiltonian cycles in networks}

The ring structure is important for distributed computing, and one may see~\cite{LHJ} for its benefits. The Hamiltonian cycles can provide an advantage for algorithms using ring structures~\cite{RB}. A cycle in a graph $\G$ is {\em Hamiltonian} if it contains every vertex of $\G$. Two Hamiltonian cycles in a graph are {\em edge-disjoint} if they have no common edge. For the convenience in writing, an edge-disjoint Hamiltonian cycle is abbreviated to an EDHC.
EDHCs can be applied on the problem of all-to-all communication algorithm. If a network contains $d$ EDHCs, then the time complexity of the algorithm can be improved by a factor of $d$. One may see~\cite{HBA,RB} for the detail. Moreover, EDHCs are also useful in fault-tolerant routing. A network  can tolerate a large number of edge failures if it has more EDHCs. When faults occur to edges of a Hamiltonian cycle, then vertices can communicate with any other vertices along another Hamiltonian cycle~\cite{HBA}.

How to search for EDHCs in a network is
an active and popular filed in the literature.
A number of networks having at least two EDHCs have been investigated, including the transposition network and some hypercube-like networks~\cite{Hung2}, the Eisenstein-Jacobi network~\cite{HBA}, and the balanced hypercube~\cite{LW}. For various results and constructions of EDHCs in networks, we refer the reader to \cite{AB, Hung,Hu} and all the references therein.
In the end of this paper, we aim to apply the symmetric property to the search of EDHCs in the spined cube, and finally, through using the 4-Cayley graphic structure of the spine cube and EDHCs in hypercubes, we prove that there exist two EDHCs in $SQ_n$ when $n\geq4$.

\medskip

The layout of this paper is as follows: In Section 2, some definitions and notations are introduced. We then discuss the symmetric property of the spined cube in Section 3,
and construct two EDHCs in the spined cube in Section 4. Section 5 is the conclusion.

\section{Preliminaries}

All graphs in this paper are finite, simple and undirected.

\subsection{Fundamental graph and group terminologies}

We follow~\cite{BM,CMY} for some terminology and definitions related to graphs and groups.
Some notations are listed in Table~\ref{table=2}.

\begin{table}[h]

\begin{center}
\begin{tabular}{l|l}

\hline
Notations & Meaning \\
\hline
$\G$& A graph with vertex set $V(\G)$ and edge set $E(\G)$\\
\hline
$N_{\G}(u)$ & The neighborhood of the vertex $u$ in $\G$\\
\hline
$d_{\G}(u,v)$ & The distance between the vertices $u$ and $v$ in $\G$\\
\hline
$I_m$& $\{1,2,\ldots,m\}$\\
\hline
$\mz_n$& A cyclic group of order $n$\\
\hline
$\mz_2^n$& An elementary abelian $2$-group of order $2^n$\\
\hline
$M\times N$ & The product of groups $M$ and $N$\\
\hline
$M\rtimes N$& A semi-product of groups $M$ and $N$\\
\hline
$1_H$& The identity of a group $H$\\
\hline
$\sg{a_1,\ldots,a_n}$& The group generated by $\{a_1,\ldots,a_n\}$\\
\hline
$G\cong H$& The groups $G$ and $H$ are isomorphic\\
\hline
\end{tabular}
\end{center}
\vskip 0cm
\caption{{\small Some notations.}}\label{table=2}
\end{table}

Let $\G$ be a graph, and let $F$ be a subset of $V(\G)$.
The subgraph of $\G$ induced by $F$ is the graph whose vertex set is $F$ and edge set is $\{(u, v)\in E(\G)~|~u,v\in F\}$, denoted by $\G[F]$.
The notation $\G-F$ represents the subgraph of $\G$ after deleting all vertices and edges in $\G[F]$ from $\G$.

Let $P_{1}=(x_{1},x_{2},\ldots,x_{m})$ and $P_{2}=(y_{1},y_{2},\ldots,y_{n})$ be two paths in a graph $\G$ such that all vertices in $V(P_1)\cup V(P_2)$ are all distinct except $x_{m}=y_{1}$. One can use $P_{1}+P_{2}$ to denote the {\it path-concatenation} of $P_{1}$ and $P_{2}$ as the path $(x_{1},x_{2},\ldots,x_{m},y_{2},\ldots,y_{t})$,
and use $P_1-(x_1,x_2)$ to denote the path $(x_{2},x_{3},\ldots,x_{m})$.

For two graphs $\G$ and $\Sigma$, an {\em isomorphism} from $\G$ to $\Sigma$ is a bijection $\phi: V(\G)\rightarrow V(\Sigma)$ such that $(u, v)\in E(\G)$ if and only if $(u^{\phi},v^{\phi})\in E(\Sigma)$. The graphs $\G$ and $\Sigma$ are {\em isomorphic}, write $\G\cong \Sigma$, if there is an isomorphism from $\G$ to $\Sigma$.

\subsection{The $m$-Cayley graph}

Let $G$ be a finite group, and let $S$ be a subset of $G$ such that $1_G\notin S$ and $S=S^{-1}=\{s^{-1}~|~s\in S\}$. The {\em Cayley graph}  of $G$ with respect to $S$,
 write $\Cay(G, S)$, is the graph with vertex set $G$ and edge set
$$\{(g, sg)~|~ g\in G, s\in S\}.$$
For the definition of $m$-Cayley graphs, we follow~\cite{HKM}.

\begin{defi}\label{defi=n-Cayley}
Let $H$ be a group, and let $T_{ij}$ be a subset of $H$ such that $T_{ij}^{-1}=T_{ji}$ and $1_H\notin T_{ii}$, where $i,j\in I_m$.
The {\em $m$-Cayley graph} of $H$ relative to the subsets $T_{ij}$s is the graph having vertex set $\{h_i : h\in H, i\in I_m\}$, and
the vertex $h_i$ is adjacent to $g_j$ if and only if there exists an element $t\in T_{ij}$ such that $g_j=(th)_j$.
\end{defi}

Clearly, a 1-Cayley graph is just a Cayley graph, and a 2-Cayley graph is also called a {\em bi-Cayley graphs}.
A bi-Cayley graph of a group $H$ relative to the subsets $T_{11},T_{22},T_{12}$ is often denoted by $\BiCay(H,T_{11},T_{22},T_{12})$.
For an $m$-Cayley graph $\G$ of a group $H$ relative to the subsets $T_{ij}$s, where $T_{ij}\subset H$ and $i,j\in I_m$, it follows from Definition~\ref{defi=n-Cayley} that the induced subgraph
by $H_i$ in $\G$ is isomorphic to the Cayley graph $\Cay(H,T_{ii})$.

\subsection{The hypercube and the spined cube}

Let $n$ be a positive integer. An {\it $n$-dimensional hypercube $Q_n$} is a graph with $2^{n}$ vertices. Each vertex is labeled with an $n$-bit binary string $x_{1}x_{2}\cdots x_{n-1}x_{n}$, where $x_i=0$ or $1$ for each $1\leq i\leq n$, and two vertices are adjacent if they have exactly one bit distinct.
The following proposition about $Q_n$ is well-known and can be also checked easily.

\begin{prop}\label{prop=2.3-1}
Let $n\geq3$ be an integer, and let $\{a_1,\ldots,a_n\}$ be a generating subset of $\mz_2^n$. Then $Q_n\cong \Cay(\mz_2^n,\{a_1,a_2,\ldots,a_n\})$.
\end{prop}

The spined cube were defined by Zhou et al.~\cite{ZFJZ} in the following way.

\begin{defi}\label{defi=1}
Let $n$ be a positive integer. An $n$-dimensional
spined cube, denoted by $SQ_n$, is defined recursively as follows:
\begin{enumerate}
\item [\rm (1)] $SQ_1$ is a complete graph on
two vertices $0$ and $1$.
\item [\rm (2)] For $n\geq 2$, $SQ_{n}$ consists of two copies of $SQ_{n-1}$, denoted by $0SQ_{n-1}$ and $1SQ_{n-1}$. Each vertex
$x=0x_2\cdots x_n$ in $0SQ_{n-1}$ connects exactly one vertex $x'$ in $1SQ_{n-1}$, where
\begin{enumerate}
\item [\rm (2.1)] $x'=1x_2$ for $n=2$;
\item [\rm (2.2)] $x'=1((x_2+x_n)~(\mod 2))x_3\cdots x_n$ for $n=3$ or $4$;
\item [\rm (2.3)] $x'=1((x_2+x_{n-1})~(\mod 2))((x_3+x_n)~(\mod2))x_4\cdots x_n$ for $n\geq5$.
\end{enumerate}
 \end{enumerate}
\end{defi}

By Definition~\ref{defi=1}, $SQ_2$ is a 4-cycle, that is, a cycle of length 4.
For notation convenience, ``$({\rm mod}~2)$" will not appear in similar expressions in the rest of the paper.


An equivalent definition of the spined cube can be obtained easily from Definition~\ref{defi=1}, as follows.

\begin{defi}\label{defi=2}  An {\it $n$-dimensional spined cube $SQ_n$} is an undirected graph with $2^{n}$ vertices with addresses $x_{1}\cdots x_{n}$, where $x_i=0$ or $1$ for each $1\leq i\leq n$. Two vertices $x=x_{1}\cdots x_{n}$ and $y$ are adjacent if and only if one of the following conditions is satisfied:

\begin{enumerate}
\item [\rm (1)]$y=(1+x_1)$ with $n=1$;
\item [\rm (2)] $y=(1+x_{1})x_2$ or $x_1(1+x_2)$ with $n=2$;
\item [\rm (3)] $y\in\{(1+x_1)(x_2+x_3)x_3,x_1(1+x_{2})x_3,x_1x_2(1+x_{3})\}$ with $n=3$;
\item [\rm (4)] $y\in\{(1+x_{1})(x_{2}+x_4)x_{3}x_4,x_1(1+x_{2})(x_{3}+x_4)x_4,x_1x_{2}(1+x_{3})x_4,
    x_1x_{2}x_{3}(1+x_4)\}$ with $n=4$;
\item [\rm (5)] $y\in\{(1+x_1)(x_2+x_{4})(x_3+x_5)x_4 x_5,x_1(1+x_{2})(x_{3}+x_5)x_{4}x_5,x_1x_{2}(1+x_{3})(x_4+x_5)x_5,x_1x_{2}x_{3}(1+x_4)x_5,x_1x_{2}x_{3}x_4(1+x_5)\}$ with $n=5$;
\item [\rm (6)] $y\in\{x_1\cdots x_{n-1}(1+x_{n}), x_1\cdots x_{n-2}(1+x_{n-1})x_n,
x_1\cdots x_{n-3}(1+x_{n-2})(x_{n-1}+x_n)x_n, x_1\cdots x_{n-4} (1+x_{n-3})(x_{n-2}+x_n)x_{n-1}x_n,  x_1\cdots x_{k-1}(1+x_k)(x_{k+1}+x_{n-1})(x_{k+2}+x_{n})x_{k+3}\cdots x_n,(1+x_1)(x_2+x_{n-1})(x_3+x_n)x_4\cdots x_n~|~2\leq k\leq n-4\}$ with $n\geq6$.
\end{enumerate}

\end{defi}

\section{Symmetric property of the spined cube}

This section is divided into two parts, in which
the $m$-Cayley property and the vertex-transitivity of the spined cube $SQ_n$ are considered, respectively. Since $SQ_2$ is a 4-cycle and this case is trivial, we always assume $n\geq3$ in this section.

\subsection{$m$-Cayley property of $SQ_n$}

For the case of $n\leq 5$, we first introduce a Cayley graph and two bi-Cayley graphs.

\begin{defi}\label{defi=3}
Let $H=\sg{a_1}\times \sg{a_2}\times \sg{a_{3}}\cong\mz_2^{3}$, $G=\sg{a}\cong\mz_8$, and $K=\sg{b_1,b_2,b_3,b_4~|~b_1^2=b_2^2=b_3^2=b_4^2=[b_1,b_2]=[b_1,b_3]=
[b_2,b_3]=[b_2,b_4]=[b_3,b_4]=1_K,b_4b_1=b_1b_3b_4}\cong\mz_2^3\rtimes \mz_2$ a non-abelian group of order $16$. Define a Cayley graph $\G_3$ and two bi-Cayley graphs $\G_4,\ \G_5$ as follows:
$$\begin{array}{ll}
&\G_{3}=\Cay(G,\{a,a^{-1},a^4\});\\
&\G_4=\BiCay(H,\{a_1,a_2,a_3\},\{a_1a_2,a_2a_3,a_3\},\{1_H\});\\
&\G_5=\BiCay(K,\{b_1,b_2,b_3,b_4\},\{b_1b_2,b_2b_3,b_2b_4,b_4\},\{1_K\}).
\end{array}$$

\end{defi}


Clearly, $\G_3$ is vertex-transitive, as it is a Cayley graph.
It can be checked easily that $SQ_3\cong\G_3$. Note that a 3-dimensional spined cube $SQ_3$ is also called a 3-dimensional locally twisted cube (see~\cite[Figure~2]{CMY}).
For $\G_4$ and $\G_5$, it can be easily checked by using the software {\sc Magma}~\cite{BCP} that $SQ_4\cong\G_4$ and $SQ_5\cong \G_5$. Moreover, neither of them is vertex-transitive.

\begin{lem}\label{lem=1}
The following hold.
\begin{enumerate}
\item [\rm (1)] $SQ_3$ is vertex-transitive and is a Cayley graph of a cyclic group $\mz_8$.
\item [\rm (2)] $SQ_4$ is a bi-Cayley graph of an elementary abelian group $\mz_2^3$, and is not vertex-transitive.
\item [\rm (3)] $SQ_5$ is a bi-Cayley graph of a non-abelian group $\mz_2^3\rtimes \mz_2$, and is not vertex-transitive.
\end{enumerate}
\end{lem}

Now, we turn to the case $n\geq6$, and we describe a family of 4-Cayley graphs of
the elementary abelian $2$-groups.

\begin{defi}\label{defi=4}
Let $n\geq6$ be an integer, and let $H=\sg{a_1}\times \sg{a_2}\times\cdots \times\sg{a_{n-2}}\cong\mz_2^{n-2}$.
Set
$$\begin{array}{ll}
&T_{11}=\{a_1,a_2,\ldots, a_{n-2}\};~T_{22}=\{a_1a_3,a_2a_4,\ldots,a_{n-4}a_{n-2},a_{n-3}a_{n-2}\};\\
&T_{33}=\{a_1a_2a_3,a_2a_3a_4,\ldots,a_{n-4}a_{n-3}a_{n-2},a_{n-3}a_{n-2}\};\\
&T_{44}=\{a_1a_2,a_2a_3,\ldots, a_{n-4}a_{n-3},a_{n-3},a_{n-2}\};\\
&T_{12}=T_{14}=T_{34}=\{1_H\};
~T_{23}=\{1_H,a_{n-2}\};~T_{13}=T_{24}=\emptyset.
\end{array}$$
Define
a $4$-Cayley graph of $H$ relative to $T_{ij}$s, and denote it by $\G_{n}$.
\end{defi}

By Definition~\ref{defi=n-Cayley}, the 4-Cayley graph $\G_{n}$
has order $2^{n}$ and valency $n$. The vertex set $V(\G_{n})$ is $\bigcup_{i=1}^4H_i$, where $H_i=\{h_i~|~h\in H\}$ for each $i\in I_4$.
The neighborhood of a vertex $h_i$ in $\G_n$ for each $h\in H$ and $i\in I_4$ are
\begin{eqnarray}
N_{\G_n}(h_1)&=&\{h_2,\ h_4,\ (a_kh)_1~|~1\leq k\leq n-2\};
\label{eq=1} \\
N_{\G_n}(h_2)&=&\{h_1,\ h_3,\ (a_{n-2}h)_3,~(a_{n-3}a_{n-2}h)_2,\ (a_ka_{k+2}h)_2 ~|~1\leq k\leq n-4\};\label{eq=2}\\
N_{\G_n}(h_3)&=&\{h_4,\ h_2,\ (a_{n-2}h)_2,\ (a_{n-3}a_{n-2}h)_3,\ (a_ka_{k+1}a_{k+2}h)_3~|~1\leq k\leq n-4\};\label{eq=3}\\
N_{\G_n}(h_4)&=&\{h_1,\ h_3,\ (a_{n-2}h)_4,\ (a_{n-3}h)_4,\ (a_ka_{k+1}h)_4~|~1\leq k\leq n-4\}.\label{eq=4}
\end{eqnarray}

We note that the induced subgraphs by $H_1$ and $H_4$ in $\G_n$ are Cayley graphs of $H\cong\mz_2^{n-2}$ with respect to $T_{11}$ and $T_{44}$, respectively (see Definition~\ref{defi=n-Cayley}). Since $H=\langle T_{11}\rangle=\langle T_{44}\rangle$, the Proposition~\ref{prop=2.3-1} implies that the two induced subgraphs are isomorphic to $Q_{n-2}$. For $i=2$ or 3, the induced subgraph by $H_i$ in $\G_n$ consists of two components, and each of them is isomorphic to $Q_{n-3}$.
Therefore, there are eight vertex-disjoint $(n-3)$-dimensional hypercubes in $\G_n$.
This observation will be also put in the end of Section 3.1 and used in Section 4.


In the following, we will prove that a spined cube $SQ_n$ with $n\geq6$ is isomorphic to $\G_n$. The following well-known fact relative to elementary abelian groups will be frequently used in the later proof.

\medskip

\f {\bf Fact} {\em For any $1\leq i, j\leq n-2$, $a_i^0=a_i^2=1_H$ and $a_ia_j=a_ja_i$. Every element in $H$ can be uniquely written as $a_1^{x_1}a_2^{x_2}\cdots a_{n-2}^{x_{n-2}}$, where $x_i\in\{0,1\}$ for $1\leq i\leq n-2$.}

\begin{lem}\label{lem=2}
For $n\geq6$, we have $SQ_n\cong \G_{n}$.
\end{lem}

\begin{proof} Define a map from $V(SQ_n)$ to $V(\G_{n})$ as following:
$$\begin{array}{ll}
\phi:&x_1\cdots x_{n-2}00\mapsto (a_1^{x_1}\cdots a_{n-2}^{x_{n-2}})_1,\\
&x_1\cdots x_{n-2}01\mapsto (a_1^{x_1}\cdots a_{n-2}^{x_{n-2}})_2,\\
&x_1\cdots x_{n-2}11\mapsto (a_1^{x_1}\cdots a_{n-2}^{x_{n-2}})_3,\\
&x_1\cdots x_{n-2}10\mapsto (a_1^{x_1}\cdots a_{n-2}^{x_{n-2}})_4.
\end{array}$$
where $x_i\in\{0, 1\}$ for $1\leq i\leq n-2$. By the Fact above, it can be checked easily that $\phi$ is a bijection. To show that $\phi$ is an isomorphism
from $SQ_n$ to $\G_{n}$, we need to show that $(x, y)\in E(SQ_n)$
if and only if $(x^{\phi}, y^{\phi})\in E(\G_{n})$. Since
$SQ_n$ and $\G_{n}$ have some valency, they have the same number of edges,
and since $\phi$ is a bijection, to finish the proof, it
suffices to show that $[N_{SQ_n}(x)]^{\phi}=N_{\G_n}(x^{\phi})$ for any
$x=x_1\cdots x_n\in V(SQ_n)$. We consider the following
four cases depending on $x_{n-1}x_n=00$, $01$, $11$ or $10$.

\medskip

\f {\bf Case~1:} $x_{n-1}x_n=00$, that is, $x=x_1\cdots x_{n-2}00$.

By Definition~\ref{defi=2}, the neighborhood of $x$ in $SQ_n$ is
$$
N_{SQ_n}(x)=\{ x_1\cdots x_{n-2}01,x_1\cdots x_{n-2}10,x_1\cdots x_{k-1}(1+x_k)x_{k+1}\cdots x_{n-2}00,
(1+x_1)x_2\cdots x_{n-2}00~|~2\leq k\leq n-2\}.
$$
Since $x^{\phi}=(a_1^{x_1}\cdots a_{n-2}^{x_{n-2}})_1$, Eq.~(\ref{eq=1}) implies that the neighborhood of $x^{\phi}$ in $\G_{n}$ is
$$\begin{array}{ll}
N_{\G_{n}}(x^{\phi})&=\{(a_1^{x_1}\cdots a_{n-2}^{x_{n-2}})_2,~
(a_1^{x_1}\cdots a_{n-2}^{x_{n-2}})_4,~
(a_k\cdot a_1^{x_1}\cdots a_{n-1}^{x_{n-1}})_1~|~1\leq k\leq n-2\}\\
&=\{(a_1^{x_1}\cdots a_{n-2}^{x_{n-2}})_2,~
(a_1^{x_1}\cdots a_{n-2}^{x_{n-2}})_4,~
(a_1^{x_1}\cdots a_{k-1}^{x_{k-1}}a_k^{1+x_k}a_{k+1}^{x_{k+1}}\cdots a_{n-2}^{x_{n-2}})_1,~\\
&~~~~~~(a_1^{x_1+1}a_{2}^{x_{2}}\cdots a_{n-2}^{x_{n-2}})_1~|~2\leq k\leq n-2\}.
\end{array}$$
An easy checking yields that $[N_{SQ_n}(x)]^{\phi}=N_{\G_n}(x^{\phi})$,
as required.

\medskip

\f {\bf Case~2:} $x_{n-1}x_n=01$, that is, $x=x_1\cdots x_{n-2}01$.

In this case, the neighborhood of $x$ in $SQ_n$ is
$$\begin{array}{ll}
N_{SQ_n}(x)=\{&x_1\cdots x_{n-2}00,~x_1\cdots x_{n-2}11,~x_1\cdots x_{n-3}(1+x_{n-2})11,\\
&x_1\cdots x_{n-4}(1+x_{n-3})(1+x_{n-2})01,~(1+x_1)x_2(1+x_3)x_4\cdots
x_{n-2}01,\\
&x_1\cdots x_{k-1}(1+x_k)x_{k+1}(1+x_{k+2})x_{k+3}\cdots x_{n-2}01~|~2\leq k\leq n-4\},
\end{array}$$
and from Eq.~(\ref{eq=2}), the neighborhood of $x^{\phi}=(a_1^{x_1}\cdots a_{n-2}^{x_{n-2}})_2$ in $\G_{n}$ is
$$\begin{array}{ll}
N_{\G_{n}}(x^{\phi})=\{&(a_1^{x_1}\cdots a_{n-2}^{x_{n-2}})_1,~
(a_1^{x_1}\cdots a_{n-2}^{x_{n-2}})_3,~(a_1^{x_1}\cdots a_{n-3}^{x_{n-3}}a_{n-2}^{1+x_{n-2}})_3,\\
&(a_1^{x_1}\cdots a_{n-4}^{x_{n-4}}a_{n-3}^{1+x_{n-3}}a_{n-2}^{1+x_{n-2}})_2,~
(a_{1}^{1+x_1}a_2^{x_2}a_3^{1+x_3}a_4^{x_4}\cdots
a_n^{x_{n-2}})_2,\\
&(a_1^{x_{1}}\cdots a_{k-1}^{x_{k-1}}a_{k}^{1+x_k}a_{k+1}^{x_{k+1}}a_{k+2}^{1+x_{k+2}}
a_{k+3}^{x_{k+3}}\cdots a_{n-2}^{x_{n-2}})_2~|~2\leq k\leq n-4\}.
\end{array}$$
Again, by an easy checking, we have $[N_{SQ_n}(x)]^{\phi}=N_{\G_n}(x^{\phi})$,
as required.

\medskip

\f {\bf Case~3:} $x_{n-1}x_n=11$, that is, $x=x_1\cdots x_{n-2}11$.

In this case, the neighborhood of $x$ in $SQ_n$ is
$$\begin{array}{ll}
N_{SQ_n}(x)=\{&x_1\cdots x_{n-2}10,~x_1\cdots x_{n-2}01,~x_1\cdots x_{n-3}(1+x_{n-2})01,\\
&x_1\cdots x_{n-4}(1+x_{n-3})(1+x_{n-2})11,~(1+x_1)(1+x_2)(1+x_3)x_4\cdots
x_{n-2}11,\\
&x_1\cdots x_{k-1}(1+x_k)(1+x_{k+1})(1+x_{k+2})x_{k+3}\cdots x_{n-2}01~|~2\leq k\leq n-4\},
\end{array}$$
and from Eq.~(\ref{eq=3}), the neighborhood of $x^{\phi}=(a_1^{x_1}\cdots a_{n-2}^{x_{n-2}})_3$ in $\G_{n}$ is
$$\begin{array}{ll}
N_{\G_{n}}(x^{\phi})=\{&(a_1^{x_1}\cdots a_{n-2}^{x_{n-2}})_4,~
(a_1^{x_1}\cdots a_{n-2}^{x_{n-2}})_2,~(a_1^{x_1}\cdots a_{n-3}^{x_{n-3}}a_{n-2}^{1+x_{n-2}})_2,\\
&(a_1^{x_1}\cdots a_{n-4}^{x_{n-4}}a_{n-3}^{1+x_{n-3}}a_{n-2}^{1+x_{n-2}})_3,~
(a_{1}^{1+x_1}a_2^{1+x_2}a_3^{1+x_3}a_4^{x_4}\cdots
a_n^{x_{n-2}})_3,\\
&(a_1^{x_{1}}\cdots a_{k-1}^{x_{k-1}}a_{k}^{1+x_k}a_{k+1}^{1+x_{k+1}}a_{k+2}^{1+x_{k+2}}
a_{k+3}^{x_{k+3}}\cdots a_{n-2}^{x_{n-2}})_3~|~2\leq k\leq n-4\}.
\end{array}$$
Hence $[N_{SQ_n}(x)]^{\phi}=N_{\G_n}(x^{\phi})$,
as required.

\medskip

\f {\bf Case~4:} $x_{n-1}x_n=10$, that is, $x=x_1\cdots x_{n-2}10$.

In this case, the neighborhood of $x$ in $SQ_n$ is
$$\begin{array}{ll}
N_{SQ_n}(x)=\{&x_1\cdots x_{n-2}11,~x_1\cdots x_{n-2}00,~x_1\cdots x_{n-3}(1+x_{n-2})10,
x_1\cdots x_{n-4}(1+x_{n-3})x_{n-2}10,\\
&(1+x_1)(1+x_2)x_3\cdots
x_{n-2}10,x_1\cdots x_{k-1}(1+x_k)(1+x_{k+1})x_{k+2}\cdots x_{n-2}10~|~2\leq k\leq n-4\},
\end{array}$$
and from Eq.~(\ref{eq=2}), the neighborhood of $x^{\phi}=(a_1^{x_1}\cdots a_{n-2}^{x_{n-2}})_4$ in $\G_{n}$ is
$$\begin{array}{ll}
N_{\G_{n}}(x^{\phi})=\{&(a_1^{x_1}\cdots a_{n-2}^{x_{n-2}})_1,~
(a_1^{x_1}\cdots a_{n-2}^{x_{n-2}})_3,~(a_1^{x_1}\cdots a_{n-3}^{x_{n-3}}a_{n-2}^{1+x_{n-2}})_4,(a_1^{x_1}\cdots a_{n-4}^{x_{n-4}}a_{n-3}^{1+x_{n-3}}a_{n-2}^{x_{n-2}})_4,\\
&(a_{1}^{1+x_1}a_2^{1+x_2}a_3^{x_3}\cdots
a_n^{x_{n-2}})_4,(a_1^{x_{1}}\cdots a_{k-1}^{x_{k-1}}a_{k}^{1+x_k}a_{k+1}^{1+x_{k+1}}a_{k+2}^{x_{k+2}}\cdots a_{n-2}^{x_{n-2}})_4~|~2\leq k\leq n-4\}.
\end{array}$$
Hence $[N_{SQ_n}(x)]^{\phi}=N_{\G_n}(x^{\phi})$,
as required.\end{proof}

In $SQ_n$ with $n\geq6$, a vertex $x=x_1\ldots x_{n-2}x_{n-1}x_n$ is said to be of {\em $x_{n-2}x_{n-1}x_{n}$-type}. In view of the proof of Lemma~\ref{lem=2}, we have the following corollary.

\begin{cor}\label{cor=1}
Let $n\geq6$. An $n$-dimensional spined cube $SQ_n$ can be decomposed to eight hypercubes $Q_{n-3}$s of dimension $n-3$ and three perfect matchings. Furthermore,
\begin{enumerate}
\item [\rm (1)] the subgraph induced by the vertices of type $000$ and $100$
$($or $010$ and $110)$ is isomorphic to $Q_{n-2}$;
\item [\rm (2)] the subgraphs induced by the vertices of type $000$, $100$, $001$,
$101$, $011$, $111$, $010$, and $110$ are all isomorphic to $Q_{n-3}$.
\end{enumerate}
\end{cor}

\subsection{Vertex-transitivity of $SQ_n$}

In this subsection, we derive in two lemmas that the spined cube $SQ_n$ is not vertex-transitive when $n\geq6$.

\begin{lem}\label{lem=3}
Let $n\geq 6$. There are exactly $\frac{n^2-5n+12}{2}$ $4$-cycles going through the vertex $(1_H)_1$
in $\G_n$, which are listed in Table~\ref{Table1}. In particular, the following hold.
\begin{enumerate}
\item [\rm (1)] There is only one $4$-cycle going through the edge $((1_H)_1, (1_H)_2)$ in $\G_n$.
\item [\rm (2)] There are exactly three $4$-cycles going through the edge $((1_H)_1,(1_H)_4)$ in $\G_n$.
\item [\rm (3)] The number of $4$-cycles going through
the edge $((1_H)_1, (a_i)_1)$ in $\G_n$ is $n-2$ for $n-3\leq i\leq n-2$, and $n-3$  for $1\leq i\leq n-4$.
\end{enumerate}
\end{lem}

\renewcommand{\arraystretch}{1.25}
\begin{table}[h]
\begin{center}
\begin{tabular}{|l|l|}
\hline
Row & $4$-cycles  \\
\hline
1&$((1_H)_1,(1_H)_2, (1_H)_3,(1_H)_4,(1_H)_1)$\\
\hline
2& $((1_H)_1,(1_H)_4,(a_{n-3})_4,(a_{n-3})_1,(1_H)_1)$\\
\hline
3&$((1_H)_1,(1_H)_4,(a_{n-2})_4,(a_{n-2})_1,(1_H)_1)$\\
\hline
4&$((1_H)_1,(a_{i})_1,(a_{i}a_j)_1,(a_j)_1,(1_H)_1)$, $1\leq j<i\leq n-2$\\
\hline
\end{tabular}
\end{center}
\vskip 0cm
\caption{All 4-cycles going through $(1_H)_1$ in $\G_n$}
\label{Table1}
\end{table}

\begin{proof} Let $C=((1_H)_1,w,u,v,(1_H)_1)$ be a 4-cycle going through $(1_H)_1$. We have $w\in N_{\G_n}((1_H)_1)$, and since
\begin{eqnarray}
N_{\G_n}((1_H)_1)=\{(1_H)_2,(1_H)_4, (a_i)_1~|~1\leq i\leq n-2\}\label{eq=3.1}
\end{eqnarray}
by Eq.~\eqref{eq=1}, we have $w=(1_H)_2$, $(1_H)_4$ or $(a_i)_1$.
Clearly, $w\neq v$ and $u\neq (1_H)_1$.

(1). Assume $w=(1_H)_2$, that is, $C=((1_H)_1,(1_H)_2,u,v,(1_H)_1)$.
It follows that $v\in N_{\G_{n}}((1_H)_1)\setminus\{(1_H)_2\}$ and $u\in (N_{\G_{n}}((1_H)_2)\cap N_{\G_n}(v))\setminus \{(1_H)_1\}$. By Eq.~\eqref{eq=3.1},
$v=(1_H)_4$ or $(a_i)_1$ with $1\leq i\leq n-2$.
Suppose $v=(a_i)_1$. By Eqs.~\eqref{eq=2} and~\eqref{eq=1}, we have
\begin{eqnarray}
N_{\G_{n}}((1_H)_2)&=&\{(1_H)_1, (1_H)_3, (a_{n-2})_3, (a_{n-3}a_{n-2})_2, (a_ka_{k+2})_2~|~1\leq k\leq n-4\};\label{eq=3.2}\\
N_{\G_{n}}((a_i)_1)&=&\{(a_i)_2, (a_i)_4, (a_ka_i)_1~|~1\leq k\leq n-2\},\label{eq=3.3}
\end{eqnarray}
implying that $(N_{\G_{n}}((1_H)_2)\cap N_{\G_n}(v))\setminus\{(1_H)_1\}=\emptyset$. A contradiction occurs.
Hence $v=(1_H)_4$. Since Eq.~\eqref{eq=4} implies that
\begin{eqnarray}
N_{\G_{n}}((1_H)_4)&=&\{(1_H)_1, (1_H)_3, (a_{n-2})_4, (a_{n-3})_4, (a_ka_{k+1})_4~|~1\leq k\leq n-4\},\label{eq=3.4}
\end{eqnarray}
we have $(N_{\G_{n}}((1_H)_2)\cap N_{\G_n}(v))\setminus \{(1_H)_1\}=\{(1_H)_3\}$.
We conclude that $u=(1_H)_3$.
The (1) holds.

(2). Assume $w=(1_H)_4$. Now, $v\in N_{\G_{n}}((1_H)_1)\setminus\{(1_H)_4\}$ and $u\in (N_{\G_{n}}((1_H)_4)\cap N_{\G_n}(v))\setminus \{(1_H)_1\}$.
By Eq.~\eqref{eq=3.1},
either $v=(1_H)_2$ or $v=(a_i)_1$ with $1\leq i\leq n-2$. For
the former case, the (1)
implies that $C=((1_H)_1,(1_H)_4,(1_H)_3,(1_H)_2,(1_H)_1)$.
For the latter case, we have $u\in (N_{\G_{n}}((a_i)_1)\cap N_{\G_n}((1_H)_4))\setminus \{(1_H)_1\}$. Since $N_{\G_{n}}((a_i)_1)\cap N_{\G_n}((1_H)_4)=\{(1_H)_1\}$ for $1\leq i\leq n-4$ and $\{(1_H)_1, (a_i)_4\}$ for $i=n-3$ or $n-2$ by Eqs.~\eqref{eq=3.3} and~\eqref{eq=3.4}, we have $i=n-3$ or $n-2$ and $u=(a_i)_4$. In this case, there are three 4-cycles going through $((1_H)_1,(1_H)_4)$ in $\G_n$, which are $((1_H)_1,(1_H)_4,(1_H)_3,(1_H)_2,(1_H)_1)$,
$((1_H)_1,(1_H)_4,(a_{n-3})_4,(a_{n-3})_1,(1_H)_1)$
and $((1_H)_1,(1_H)_4,(a_{n-2})_4,(a_{n-2})_1,(1_H)_1)$.
The (2) holds.

(3). Assume $w=(a_i)_1$ with $1\leq i\leq n-2$, that is, $C=((1_H)_1,(a_i)_1,u,v,(1_H)_1)$. Now, $v\in N_{\G_{n}}((1_H)_1)\setminus\{(a_i)_1\}$ and
$u\in (N_{\G_{n}}((a_i)_1)\cap N_{\G_n}(v))\setminus \{(1_H)_1\}$. It follows from Eq.~\eqref{eq=3.1} that $v=(1_H)_2$, $(1_H)_4$ or $(a_j)_1$ with $1\leq j\neq i\leq n-2$.

By (1), we have $v\neq (1_H)_2$.
If $v=(1_H)_4$, then the (2) implies that $C=((1_H)_1,(a_{n-3})_1,(a_{n-3})_4,(1_H)_4,(1_H)_1)$
or $((1_H)_1,(a_{n-2})_1,(a_{n-2})_4,(1_H)_4,(1_H)_1)$.
Finally, let $v=(a_j)_1$ with $1\leq j\neq i\leq n-2$. Now, $u\in (N_{\G_{n}}((a_i)_1)\cap N_{\G_n}((a_j)_1))\setminus\{(1_H)_1\}$, and then by Eq.~\eqref{eq=3.3},
 we have that  $u=(a_ia_j)_1$ and $C=((1_H)_1,(a_i)_1,(a_ia_j)_1,(a_j)_1,(1_H)_1)$.
Hence the number of 4-cycles going through $((1_H)_1,(a_i)_1)$ is $n-2$ when $i=n-3$ or $n-2$, and $n-3$ when $1\leq i\leq n-4$. The (3) holds.

Summing up, there are $3+\frac{(n-3)(n-2)}{2}=\frac{n^2-5n+12}{2}$ 4-cycles going through $(1_H)_1$
in $\G_n$ in total, and all of them are listed in Table~\ref{Table1}.
\end{proof}

\begin{lem}\label{lem=4}
Let $n\geq 6$. There are exactly $\frac{n^2-7n+20}{2}$ $4$-cycles going through the vertex $(1_H)_2$ in $\G_n$, which are listed in Table~\ref{Table2}.
\end{lem}

\renewcommand{\arraystretch}{1.25}
\begin{table}[h]
\begin{center}
\begin{tabular}{|l|l|}
\hline
Row & $4$-cycles  \\
\hline
1&$((1_H)_1,(1_H)_2, (1_H)_3,(1_H)_4,(1_H)_1)$\\
\hline
2&$((1_H)_2,(1_H)_3, (a_{n-2})_2,(a_{n-2})_3,(1_H)_2)$\\
\hline
3& $((1_H)_2,(1_H)_3,(a_{n-3}a_{n-2})_3,(a_{n-3}a_{n-2})_2,(1_H)_2))$\\
\hline
4&$(1_H)_2,(a_{n-2})_3,(a_{n-3})_3,(a_{n-3}a_{n-2})_2,(1_H)_2))$\\
\hline
5&$((1_H)_2,(a_{n-3}a_{n-2})_2,(a_{n-3}a_{n-2}a_{i}a_{i+2})_2,(a_{i}a_{i+2})_2,(1_H)_2)),~1\leq i\leq n-4$\\
\hline
6&$((1_H)_2,(a_{i}a_{i+2})_2,(a_{j}a_{j+2}a_{i}a_{i+2})_2,(a_{j}a_{j+2})_2,(1_H)_2))$,~$1\leq j<i\leq n-4$.\\
\hline
\end{tabular}
\end{center}
\vskip 0cm
\caption{All 4-cycles going through $(1_H)_2$ in $\G_n$}
\label{Table2}
\end{table}

\begin{proof} Let $C=((1_H)_2,w,u,v,(1_H)_2)$ be a 4-cycle going through $(1_H)_2$ in $\G_n$. Now, $w,v\in N_{\G_n}((1_H)_2)=\{(1_H)_1, (1_H)_3, (a_{n-2})_3, (a_{n-3}a_{n-2})_2, (a_ka_{k+2})_2~|~1\leq k\leq n-4\}$
(see Eq.~\eqref{eq=3.2}), $w\neq v$ and $u\neq (1_H)_2$.
If $w=(1_H)_1$ or $v=(1_H)_1$, then Lemma~\ref{lem=3}~(1) implies that $C=((1_H)_2,(1_H)_1,(1_H)_4,(1_H)_3,(1_H)_2)$.

Assume $w,v\in \{(1_H)_3, (a_{n-2})_3, (a_{n-3}a_{n-2})_2, (a_ia_{i+2})_2\}$ for some $1\leq i\leq n-4$.
Note that $u\in (N_{\G_n}(w)\cap N_{\G_n}(v))\setminus\{(1_H)_2\}$.
For all possible $w$ and $v$, we list their neighborhoods. By Eq.~\eqref{eq=2} we have
{\small
\begin{eqnarray}
N_{\G_{n}}((1_H)_3)&=&\{(1_H)_4, (1_H)_2, (a_{n-2})_2, (a_{n-3}a_{n-2})_3, (a_ka_{k+1}a_{k+2})_3~|~1\leq k\leq n-4\};\label{eq=4.1}\\
N_{\G_{n}}((a_{n-2})_3)&=&\{(a_{n-2})_4, (a_{n-2})_2, (1_H)_2, (a_{n-3})_3, (a_ka_{k+1}a_{k+2}a_{n-2})_3~|~1\leq k\leq n-4\};\label{eq=4.2}
\end{eqnarray}
}
and by Eq.~\eqref{eq=3} we have
{\small
$$\begin{array}{ll}
&N_{\G_{n}}((a_{n-3}a_{n-2})_2)=\{(a_{n-3}a_{n-2})_1, (a_{n-3}a_{n-2})_3, (a_{n-3})_3, (1_H)_2, (a_ka_{k+2}a_{n-3}a_{n-2})_2|~1\leq k\leq n-4\};\\
&N_{\G_{n}}((a_{i}a_{i+2})_2)=\{(a_{i}a_{i+2})_1, (a_{i}a_{i+2})_3, (a_{n-2}a_{i}a_{i+2})_3, (a_{n-3}a_{n-2}a_{i}a_{i+2})_2,(a_ka_{k+2}a_{i}a_{i+2})_2~|~1\leq k\leq n-4\},
\end{array}$$
}
where $1\leq i\leq n-4$. By an easy check, we have
$$\begin{array}{ll}
&N_{\G_n}((1_H)_3)\cap N_{\G_n}((a_{n-2})_3)=\{(1_H)_2,(a_{n-2})_2\};\\
&N_{\G_n}((1_H)_3)\cap N_{\G_n}((a_{n-3}a_{n-2})_2)=\{(1_H)_2,(a_{n-3}a_{n-2})_3\},\\
&N_{\G_n}((1_H)_3)\cap N_{\G_n}((a_{i}a_{i+2})_2)=\{(1_H)_2\},\\
&N_{\G_n}((a_{n-2})_3)\cap N_{\G_n}((a_{n-3}a_{n-2})_2)=\{(1_H)_2,(a_{n-3})_3\},\\
&N_{\G_n}((a_{n-2})_3)\cap N_{\G_n}((a_{i}a_{i+2})_2)=\{(1_H)_2\},\\
&N_{\G_n}((a_{n-3}a_{n-2})_2)\cap N_{\G_n}((a_{i}a_{i+2})_2)=\{(1_H)_2,(a_{n-3}a_{n-2}a_{i}a_{i+2})_2\},\\
&N_{\G_n}((a_{i}a_{i+2})_2)\cap N_{\G_n}((a_{j}a_{j+2})_2)=\{(1_H)_2,(a_{j}a_{j+2}a_{i}a_{i+2})_2\}~{\rm with}~1\leq i\neq j\leq n-4.
\end{array}$$
Hence when $w,v\neq(1_H)_1$, the 4-cycles going through $(1_H)_2$ are:
$$\begin{array}{ll}
&((1_H)_2,(1_H)_3,(a_{n-2})_2,(a_{n-2})_3,(1_H)_2));\\
&((1_H)_2,(1_H)_3,(a_{n-3}a_{n-2})_3,(a_{n-3}a_{n-2})_2,(1_H)_2));\\
&((1_H)_2,(a_{n-2})_3,(a_{n-3})_3,(a_{n-3}a_{n-2})_2,(1_H)_2));\\
&((1_H)_2,(a_{n-3}a_{n-2})_2,(a_{n-3}a_{n-2}a_{i}a_{i+2})_2,(a_{i}a_{i+2})_2,(1_H)_2)),~1\leq i\leq n-4;\\
&((1_H)_2,(a_{i}a_{i+2})_2,(a_{j}a_{j+2}a_{i}a_{i+2})_2,(a_{j}a_{j+2})_2,(1_H)_2)),~1\leq j<i\leq n-4.
\end{array}$$

Summing up, there are exactly $\frac{n^2-7n+20}{2}$ 4-cycles passing through $(1_H)_2$.\end{proof}

Clearly, $\frac{n^2-5n+12}{2}\neq \frac{n^2-7n+20}{2}$ when $n\geq6$. It follows from Lemmas~\ref{lem=3} and \ref{lem=4} that there are different number of 4-cycles going through the vertices $(1_H)_1$ and $(1_H)_2$ in $\G_n$, and so $\G_n$ is not vertex-transitive when $n\geq 6$. Combined with Lemmas~\ref{lem=1} and~\ref{lem=2}, we have the following theorem.

\begin{theorem}\label{the=4.2}
The $n$-dimensional spined cube $SQ_n$ is not vertex-transitive unless $n\leq 3$.
\end{theorem}

\section{Edge-disjoint Hamiltonian cycles in $SQ_n$}

In this section, we aim to prove that there exist
two EDHCs in the spined cube $SQ_n$ with $n\geq4$.

\begin{theorem}
There exist two EDHCs in $SQ_n$ when $4\leq n\leq 6$.
\end{theorem}

\begin{proof} By Definition~\ref{defi=1}, it can be easily checked that
the $C_n^{(1)}$ and $C_n^{(2)}$ listed below are two EDHCs in $SQ_n$ for $4\leq n\leq6$.\end{proof}

{\small
\begin{eqnarray*}
& C_4^{(1)}= &(0000,0010,1010,1011, 1101,1111,0011,0001,0111, 0110, 1110,1100,0100,0101,1001,1000,0000);\\
& C_4^{(2)}= &(0000,0100,0110,0010, 0011,0101,0111,1011,1001, 1111, 1110,1010,1000,1100,1101,0001,0000);\\
& C_5^{(1)}= &(00000,00001,00011,00010, 00110,00100,00101,00111,01011, 01001,01000,01010,\\
& &01110,01100,01101,01111,10011,10001,10111,10101,10100,10110,10010,11010,\\
& &11011,11101,11100,11110,11111,11001,11000,10000,00000);\\
& C_5^{(2)}= &(00000,00010,01010,01011, 01101,00001,00111,00110,01110, 01111, 00011,11111,\\
& &11101,01001,00101,10001,10000,10010,10011,10101,11001,11011,10111,10110,\\
& &11110,11010,11000,01000,01100,11100,10100,00100,00000);\\
& C_6^{(1)}= &(000000,100000,110000,111000,111010,111011,111101,111100,111110,111111,111001,\\
& &110101,110100,100100,100110,100111,101011,101101,101111,110011,110001,110111,\\
& &110110,110010,101010,101110,101100,001100,001101,100101,101001,101000,001000,\\
& &011000,011100,011101,011111,011110,011010,000010,000110,001110,001010,001011,\\
& &001001,100001,100011,100010,010010,010011,001111,000011,000001,000111,011011,\\
& &011001,010101,010111,010110,010100,010000,010001,000101,000100,000000);\\
& C_6^{(2)}= &(000000,010000,110000,110100,111100,111000,111001,111011,110111,110101,110011,\\
& &111111,111101,101000,100101,100111,100001,101101,101100,101000,101010,101011,\\
& &101001,101111,100011,011011,011101,010001,010111,001011,001101,011001,011000,\\
& &011010,010010,010110,100110,100010,100000,100100,000100,010100,011100,001100,\\
& &001110,001111,001001,001000,001010,000010,110010,111010,111110,110110,101110,\\
& &011110,000110,000111,000101,000011,011111,010011,010101,000001,000000).
\end{eqnarray*}
}

Next, we begin to consider the Hamiltonian cycles in $SQ_n$ when $n\geq7$.
Since an $n$-dimensional spined cube $SQ_n$ can be decomposed to eight vertex-disjoint $(n-3)$-dimensional hypercubes $Q_{n-3}$s when $n\geq 7$ by Corollary~\ref{cor=1},
our main strategy to construct EDHCs in $SQ_n$ consists of the following three steps:
\begin{itemize}
\item Step~1: Find two EDHCs in some $Q_{n-3}$s;
\item Step~2: Find EDHCs in the other $Q_{n-3}$s under graph isomorphisms;
\item Step~3: Concatenate the cycles.
\end{itemize}
In the first step, we have the following lemma about hypercubes.
For notation convenience, we use
the elements of the group $\mz_2^n$ to denote the vertices of the hypercube $Q_n$ (see Proposition~\ref{prop=2.3-1}) and the spined cube $SQ_n$ (see Lemma~\ref{lem=2}).

\begin{lem}\label{lem=4.1}
Let $n\geq 4$, and let $\{a_1,\ldots,a_n\}$ be a generating subset of $H=\mz_2^n$. In an $n$-dimensional hypercube $Q_n=\Cay(H,\{a_1,\ldots,a_n\})$,
there exist two EDHCs containing the edges $(1_H,a_n)$ and $(a_n,a_{n-1}a_n)$, respectively.
\end{lem}

\begin{proof} We proceed by induction on $n$.
It can be easily checked that the following two cycles $C_1$ and $C_2$ are EDHCs in $Q_4$ containing the edges $(1_H,a_4)$ and $(a_4,a_3a_{4})$, respectively, and so the lemma is true when $n=4$.
\begin{eqnarray*}
& C_1= &(1_H,a_3,a_2a_3,a_1a_2a_3, a_1a_3,a_1,a_1a_2
a_2,a_2a_4,a_1a_2a_4, a_1a_2a_3a_4, a_2a_3a_4,a_3a_4,a_1a_3a_4,a_1a_4,a_4,1_H);\\
& C_2= & (1_H,a_1,a_1a_4,a_1a_2a_4,a_1a_2,a_1a_2a_3,a_1a_2a_3a_4,
a_1a_3a_4,a_1a_3,a_3,a_3a_4,a_4,a_2a_4,a_2a_3a_4,a_2a_3,a_2,1_H).
\end{eqnarray*}

Assume that the lemma is true for some $k\geq4$,
and let $n=k+1$. Let
$H_1=\sg{a_1,\ldots,a_{k}}$, the subgroup of $\mz_2^{k+1}$
generated by $a_1,\ldots, a_{k}$, and denote $H_2=\{ga_{k+1}~|~g\in H_1\}$, the coset of
$H_1$ in $\mz_2^{k+1}$. Clearly, $H_1\cong\mz_2^{k}$, and $Q_{k+1}[H_1]\cong Q_{k+1}[H_2]\cong Q_{k}$ (see Proposition~\ref{prop=2.3-1}). The map $g\mapsto ga_{k+1},~\forall g\in H_1,$ induces an isomorphism from
$Q_{k+1}[H_1]$ to $Q_{k+1}[H_1]$, say $\alpha$. Moreover, for each vertex $g$ in $Q_{k+1}[H_1]$, $g^{\a}=ga_{k+1}$ is the unique neighbor of $g$ in $Q_{k+1}[H_2]$.

By the induction hypothesis, there exist two EDHCs
$C_1$ and $C_2$, containing the edges $(1_H,a_{k})$ and $(a_{k},a_{k-1}a_{k})$, respectively. Since $\alpha$ is an isomorphism from $Q_{k+1}[H_1]$ to $Q_{k+1}[H_1]$,
$C_1^{\alpha}$ and $C_2^{\alpha}$ are two EDHCs
in $Q_{k+1}[H_2]$. Moreover, the cycles $C_1^{\alpha}$ contains the edge $(1_H,a_{k})^{\alpha}=(a_{k+1},a_{k}a_{k+1})$.
Let
$$\widehat{C_1}=C_1-(1_H,a_{k})+(1_H,a_{k+1})+C_1^{\a}-(a_{k+1},a_{k}a_{k+1})+(a_{k}a_{k+1},a_{k}).$$
Now, $\widehat{C_1}$ is an Hamiltonian cycle in $Q_{k+1}$, and the edge $(1_H,a_{k+1})$ belongs to $\widehat{C_1}$.


Since $C_1$ and $C_2$ are edge-disjoint and $(1_H,a_{k})\in E(C_1)$,
the edge $(1_H,a_{k})\notin E(C_2)$, that is, $d_{C_2}(1,a_{k})\geq2$.
Noting that $C_2$ is a Hamiltonian cycle in $Q_n[H_1]$
with length $2^k\geq 2^4$, we may assume that
$C_2=(1_H,u,\ldots,a_{k},v,\ldots,1_H)$,
where $u$ and $v$ are neighbors of $1_H$ and $a_{k}$ in $C_2$, respectively.
Clearly,  $\{u,v\}\cap \{a_{k+1},a_{k},1_H\}=\emptyset$.
Denote $P_1$ be the subpath from $u$ to $a_{k}$ in $C_2$,
and $P_2$ the subpath from $1$ to $v$.
Let
\begin{eqnarray*}
& \widehat{C_2}& =P_1+(a_{k},1_H)+P_2+(v,v^{\a})+P_2^{\a}+(1_H^{\a},a_{k}^{\a})+P_1^{\a}+(u^{\a},u)\\
& &=P_1+(a_{k},1_H)+P_2+(v,va_{k+1})+P_2^{\a}+(a_{k+1},a_{k+1}a_{k})+P_1^{\a}+(ua_{k+1},u).
\end{eqnarray*}
The $\widehat{C_2}$ is an Hamiltonian cycle in $Q_n$, containing the edge $(a_{k+1},a_{k+1}a_{k})$. Since $v\notin\{1_H,a_{k}\}$,
we have $\{(1_H,a_{k+1}),(a_{k}a_{k+1},a_{k})\}\cap \{(v,va_{k+1}), (a_{k+1},a_{k+1}a_{k})\}=\emptyset$,
and since $E(C_1)\cap [E(P_1)\cap E(P_2)]\subseteq E(C_1)\cap E(C_2)=\emptyset$,
we conclude that $\widehat{C_1}$ and $\widehat{C_2}$ are edge-disjoint. The proof is complete.\end{proof}

Next, we aim to find some graph isomorphisms (see Step 2), and we need some symbols. Let $H=\sg{a_1}\times\cdots \times\sg{a_{n-2}}\cong\mz_2^{n-2}$ with $n\geq 7$,
and let
\begin{eqnarray*}
H^{11}& =&\sg{a_1,a_2,\ldots, a_{n-3}};\\
H^{21}&=&\sg{a_1a_3,a_2a_4,\ldots,a_{n-4}a_{n-2},a_{n-3}a_{n-2}};\\
H^{31}&=&\sg{a_1a_2a_3,a_2a_3a_4,\ldots,a_{n-4}a_{n-3}a_{n-2},a_{n-3}a_{n-2}};\\
H^{41}&=&\sg{a_1a_2,a_2a_3,\ldots, a_{n-4}a_{n-3},a_{n-3}};\\
H^{i2}&=&\{ha_{n-2}~|~h\in H^{i1}\},~1\leq i\leq 4;\\
H_{i}^{ij}&=&\{h_i~|~h\in H^{ij}\},~1\leq i\leq 4,~1\leq j\leq 2.
\end{eqnarray*}
By some primary knowledge in group theory, one can observe that $H^{i1}$ is a subgroup of $H$ isomorphic to $\mz_2^{n-3}$, $H^{i2}$ is a coset of $H^{i1}$ in $H$, and $H=H^{i1}\cup H^{i2}$, where $1\leq i\leq 4$.
In view of Lemma~\ref{lem=2}, $SQ_n$ is a 4-Cayley graph of $H$, and by Definition~\ref{defi=3}, we have $V(SQ_n)=\bigcup_{i=1}^4H_i=\bigcup_{i=1}^4(H_i^{i1}\cup H_i^{i2})$. Moreover, the induced subgraph $SQ_n[H_i^{ij}]$ is isomorphic to $Q_{n-3}$
(see Corollary~\ref{cor=1}), where $1\leq i\leq 4$ and $1\leq j\leq2$.
Define the map from $H_i^{i1}$ to $H_i^{i2}$ with $1\leq i\leq 4$ as follows:
\begin{eqnarray*}
&\R(a_{n-2}): &h_i\mapsto (ha_{n-2})_i,~\forall h_i\in H_i^{i1}.
\end{eqnarray*}
It can be checked easily that $\R(a_{n-2})$ is an isomorphism from $SQ_n[H_i^{i1}]$ to $SQ_n[H_i^{i2}]$.

Now, we are ready to construct EDHCs in the spined cube $SQ_n$
with $n\geq7$.

\begin{theorem}\label{theo=4.1}
There exist two EDHCs in $SQ_n$ when $n\geq 7$.
\end{theorem}

\begin{proof} Since the induced subgraphs
\begin{eqnarray*}
SQ_n[H_1^{11}]& \cong &\Cay(H^{11},\{a_1,a_2,\ldots, a_{n-3}\})\cong Q_{n-3};\\
SQ_n[H_2^{21}]&\cong&\Cay(H^{21},\{a_1a_3,a_2a_4,\ldots,a_{n-4}a_{n-2},a_{n-3}a_{n-2}\})\cong Q_{n-3};\\
SQ_n[H_3^{31}]&\cong&\Cay(H^{31},\{a_1a_2a_3,a_2a_3a_4,\ldots,a_{n-4}a_{n-3}a_{n-2},a_{n-3}a_{n-2}\})\cong Q_{n-3};\\
SQ_n[H_4^{41}]&\cong&\Cay(H^{41},\{a_1a_2,a_2a_3,\ldots, a_{n-4}a_{n-3},a_{n-3}\})\cong Q_{n-3},
\end{eqnarray*}
each of them admits two EDHCs by Lemma~\ref{lem=4.1}.
Assume that $C_{i1}$ and $C_{i2}$ are two EDHCs in
$SQ_n[H_i^{i1}]$, and by Lemma~\ref{lem=4.1} we may further assume that
\begin{eqnarray*}
&&((1_H)_1,(a_{n-3})_1)\in E(C_{11}),~((a_{n-3})_1, (a_{n-4}a_{n-3})_1)\in E(C_{12});\\
&&((1_H)_2,(a_{n-3}a_{n-2})_2)\in E(C_{21}),~ ((a_{n-3}a_{n-2})_2,(a_{n-4}a_{n-3})_2)\in E(C_{22});\\
&&((1_H)_3,(a_{n-3}a_{n-2})_3)\in E(C_{31}),~ ((a_{n-3}a_{n-2})_3,(a_{n-4})_3)\in E(C_{32});\\
&&((1_H)_4,(a_{n-3})_4)\in E(C_{41}),~ ((a_{n-3})_4,(a_{n-4})_4)\in E(C_{42}).
\end{eqnarray*}


Let $C_{i1}'=C_{i1}^{\R(a_{n-2})}$ and $C_{i2}'=C_{i2}^{\R(a_{n-2})}$ for $1\leq i\leq 4$. Since $\R(a_{n-2})$ is an isomorphism from $SQ_n[H_i^{i1}]$ to $SQ_n[H_i^{i2}]$, $C_{i1}'$ and $C_{i2}'$ are two EDHCs of $SQ_n[H_i^{i2}]$.
Since
$((1_H)_1,(a_{n-3})_1)\in E(C_{11})$ and $((a_{n-3})_1, (a_{n-4}a_{n-3})_1)\in E(C_{12})$,
we have
\begin{eqnarray*}
&&((1_H)_1,(a_{n-3})_1)^{\R(a_{n-2})}=((a_{n-2})_1,(a_{n-3}a_{n-2})_1)\in E(C_{11}'),\\
&&((a_{n-3})_1, (a_{n-4}a_{n-3})_1)^{\R(a_{n-2})}=((a_{n-3}a_{n-2})_1,(a_{n-4}a_{n-3}a_{n-2})_1 )\in E(C_{12})'.
\end{eqnarray*}
Similarly, we have
\begin{eqnarray*}
&&((a_{n-2})_2,(a_{n-3})_2)\in E(C_{21}'),~ ((a_{n-3})_2,(a_{n-4}a_{n-3}a_{n-2})_2)\in E(C_{22}');\\
&&((a_{n-2})_3,(a_{n-3})_3)\in E(C_{31}'),~ ((a_{n-3})_3,(a_{n-4}a_{n-2})_3)\in E(C_{32}');\\
&&((a_{n-2})_4,(a_{n-3}a_{n-2})_4)\in E(C_{41}'),~ ((a_{n-3}a_{n-2})_4,(a_{n-4}a_{n-2})_4)\in E(C_{42'}).
\end{eqnarray*}

Now, we have 16 cycles in $SQ_n$, and clearly, they are edge-disjoint.
Finally, we concatenate the cycles.
Let
\begin{eqnarray*}
&C_1= & C_{11}-((1_H)_1,(a_{n-3})_1)+((a_{n-3})_1,(a_{n-3}a_{n-2})_1)+C_{11}'-((a_{n-3}a_{n-2})_1,(a_{n-2})_1)\\
&&+((a_{n-2})_1,(a_{n-2})_2)+C_{21}'-((a_{n-2})_2,(a_{n-3})_2)+((a_{n-3})_2,(a_{n-3}a_{n-2})_3)
+C_{31}\\
&&-((a_{n-3}a_{n-2})_3,(1_H)_3)+((1_H)_3,(1_H)_2)+C_{21}-((1_H)_2,(a_{n-3}a_{n-2})_2)\\
&&+((a_{n-3}a_{n-2})_2,(a_{n-3})_3)+C_{31}'-((a_{n-3})_3,(a_{n-2})_3)+((a_{n-2})_3,(a_{n-2})_4)\\
&&+C_{41}'-((a_{n-2})_4,(a_{n-3}a_{n-2})_4)+((a_{n-3}a_{n-2})_4,(a_{n-3})_4)+C_{41}\\
&&-((a_{n-3})_4,(1_H)_4)+((1_H)_4,(1_H)_1),
\end{eqnarray*}
and let
\begin{eqnarray*}
&C_2= &C_{12}-((a_{n-3})_1, (a_{n-4}a_{n-3})_1)+((a_{n-4}a_{n-3})_1,(a_{n-4}a_{n-3})_2)+C_{22}\\
&&-((a_{n-3}a_{n-2})_2,(a_{n-4}a_{n-3})_2)+((a_{n-3}a_{n-2})_2,(a_{n-3}a_{n-2})_1)
+C_{12}'\\
&&-((a_{n-3}a_{n-2})_1,(a_{n-4}a_{n-3}a_{n-2})_1)+((a_{n-4}a_{n-3}a_{n-2})_1,(a_{n-4}a_{n-3}a_{n-2})_2)+C_{22}'\\
&&-((a_{n-4}a_{n-3}a_{n-2})_2,(a_{n-3})_2)+((a_{n-3})_2,(a_{n-3})_3)+C_{32}'-((a_{n-3})_3,(a_{n-4}a_{n-2})_3)\\
&&+((a_{n-4}a_{n-2})_3,(a_{n-4}a_{n-2})_4)+C_{42}'-((a_{n-3}a_{n-2})_4,(a_{n-4}a_{n-2})_4)\\
&&+((a_{n-3}a_{n-2})_4,(a_{n-3}a_{n-2})_3)+C_{32}-((a_{n-3}a_{n-2})_3,(a_{n-4})_3)\\
&&+((a_{n-4})_3,(a_{n-4})_4)+C_{42}-((a_{n-4})_4,(a_{n-3})_4)+((a_{n-3})_4,(a_{n-3})_1).
\end{eqnarray*}
We conclude that $C_{1}$ and $C_{2}$ are EDHCs of $SQ_n$. The proof is complete.\end{proof}

\section{Conclusion}

Graph symmetry is an important factor in the design of a network. The spined cube $SQ_n$ was introduced by Zhou et al.~\cite{ZFJZ} in 2011 as a variant of the hypercube $Q_n$, whose diameter is less than most known variants of hypercubes.
The hypercube have been well-studied in the literature. A natural problem is how to use the numerous works about hypercubes to study the variants,
that is, how to establish the connection between hypercube and its variants.
This is also a reason why we consider the symmetric property of
the spined cube in this paper.
We first prove that
$SQ_n$ is a 4-Cayley graph of an elementary abelian 2-group $\mz_2^{n-2}$ when $n\geq6$,
and then have that it is not vertex-transitive unless $n\leq3$.
The symmetric property of $SQ_n$ shows that
it can be decomposed to eight vertex-disjoint $(n-3)$-dimensional hypercubes
when $n\geq6$. By using the existence of EDHCs in hypercubes, we show that there exists two EDHCs in $SQ_n$ when $n\geq4$.

\medskip

\f {\bf Acknowledgement:} The second author was supported by the National Natural Science Foundation of China (12101070,11731002,1201101021,12161141005).
The third author was supported by the National Natural Science Foundation of China (11731002,1201101021,12161141005).

\end{document}